\documentclass[11pt,reqno]{article}
\usepackage{amssymb}
\usepackage{amsmath}
\usepackage{amsthm}
\usepackage{amsfonts}
\usepackage{url}
\usepackage{breakurl}
\usepackage{hyperref}
\usepackage{comment}

\renewcommand{\leq}{\leqslant}
\renewcommand{\ge}{\geqslant}
\renewcommand{\le}{\leqslant}

\newcommand{\R}{{\mathbb R}}

\newcommand{\li}{{\textrm{li}}}		
\newcommand{\sump}{\sideset{}{'}\sum}	
\newcommand{\sumpsmall}{\sideset{}{'}{\textstyle \sum}} 
\newcommand{\NZ}{{\mathcal F}}		
\newcommand{\ub}{\underbar}

\newcommand{\dif}{{\,d}}        
\theoremstyle{plain}
\newtheorem{theorem}{Theorem}
\newtheorem{corollary}{Corollary}
\newtheorem{lemma}{Lemma}

\theoremstyle{definition}

\newtheorem{remark}{Remark}

\newtheorem{example}{Example}

\begin{document}
\bibliographystyle{plain}
\title{Accurate estimation of sums over zeros of the Riemann zeta-function%
}
\author
{Richard P.\ Brent\footnote{Australian National University,
Canberra, Australia
{\tt <accel@rpbrent.com>}},\; 
David J.\ Platt\footnote{School of Mathematics, University of Bristol,
Bristol, UK 
{\tt <dave.platt@bris.ac.uk>}}\;
and Timothy S.\ Trudgian\footnote{School of Science, Univ.\ of NSW,
Canberra, Australia 
{\tt <t.trudgian@adfa.edu.au>}}
}
\date{\today}	
\maketitle

\vspace*{-10pt}
\begin{abstract}
\vspace*{5pt}
We consider sums of the form 
$\sum \phi(\gamma)$, where $\phi$ is a given function,
and $\gamma$ ranges over
the ordinates of nontrivial zeros of the Riemann zeta-function
in a given interval. We show how the numerical estimation of such
sums can be accelerated by a simple device, and give examples
involving both convergent and divergent infinite sums.
\end{abstract}

\section{Introduction}					\label{sec:Intro}

Let the nontrivial zeros of the Riemann zeta-function $\zeta(s)$ be
denoted by $\rho = \beta + i\gamma$. In order of increasing height,
the ordinates of the zeros in the upper half-plane are
$\gamma_1 \approx 14.13  < \gamma_2 < \gamma_3 < \cdots$.

Let $\phi:[T_0,\infty) \mapsto [0,\infty)$ be a non-negative function
on the interval $[T_0,\infty)$, for some $T_0 \ge 1$. 
Throughout this paper we assume 
that $\phi(t)$ is twice continuously differentiable and satisfies the
conditions $\phi'(t) \le 0$ and $\phi''(t) \ge 0$ on $[T_0,\infty)$.
These conditions imply that $\phi(t)$ is convex on $[T_0,\infty]$.

We are interested in sums of the form
$
\sumpsmall_{T_1 \le \gamma \le T_2}\phi(\gamma)
\text{ and }
\sumpsmall_{T_1 \le \gamma}\phi(\gamma),
$
where $T_0 \le T_1 \le T_2$~.
Here the prime symbol ($'$) indicates that if 
$\gamma=T_1$ or $\gamma=T_2$ 
then the term $\phi(\gamma)$ is given weight $\frac12$.
If multiple zeros exist, then
terms involving such zeros are weighted by their multiplicities.
Sums of this form can be bounded using
a lemma of Lehman~\cite[Lem.~1]{Lehman} that we state for reference.
We have changed Lehman's wording slightly, but the proof is the same.
In the lemma and elsewhere, $\vartheta$ denotes a real number in $[-1,1]$,
possibly different at each occurrence.

\begin{lemma}[Lehman]		                        \label{lem:Lehman}
If $2\pi e \le T \le T_2$ and $\phi:[T,T_2]\mapsto [0,\infty)$
is monotone non-increasing on $[T,T_2]$, then
\[
\sump_{T\le\gamma\le T_2}\!\!\phi(\gamma) =
\frac{1}{2\pi}\int_{T}^{T_2}\! \phi(t)\log(t/2\pi)\dif t
\,+\, A\vartheta\left(\!2\phi(T)\log T + 
   \int_{T}^{T_2}\frac{\phi(t)}{t}\dif t\right)\!,
\]
where $A$ is an absolute constant.\footnote{In
Lemma~\ref{lem:Lehman}, $A$ is a constant such that
$|Q(T)| \le A\log T$ for all $T \ge 2\pi e$,
where $Q(T)$ is as in~\eqref{eq:N}.
{From} \hbox{\cite[Cor.~1]{BPTCv7}}, we may take $A = 0.28$.}
\end{lemma}

Our Lemma~\ref{lem:finite} may be seen as a refinement of Lehman's
lemma, with the additional assumption that $\phi''(t) \ge 0$.
Lemma~\ref{lem:finite} is stated and proved in \S\ref{sec:finite}.
For simplicity we outline here the case $T_2 \to \infty$, since this case
has one fewer parameter and is of interest in many applications.

If the infinite sum $\sumpsmall_{T \le \gamma}\phi(\gamma)$ converges,
then the error term in Lemma~\ref{lem:Lehman} is $\gg\phi(T)\log T$.
In Theorem~\ref{thm:convergent} we express the error as
$-\phi(T)Q(T) + E_2(T)$, where $Q(T) \ll \log T$ can be computed
from~\eqref{eq:N}--\eqref{eq:L},
and $E_2(T)$ is generally of lower order than $\phi(T)\log T$.
We state Theorem~\ref{thm:convergent} here; the proof is given in
\S\ref{sec:convergent}.
Note that the lower bound on $T$ is $2\pi$, not $2\pi e$ as in
Lehman's lemma. This is convenient in applications
because $2\pi < \gamma_1 < 2\pi e$.

\begin{theorem}					\label{thm:convergent}
Suppose that $2\pi \le T_0 \le T$ and
$\int_T^\infty\phi(t)\log(t/2\pi)\dif t < \infty$. Let
\begin{equation}				\label{eq:ET0}
E(T) := \sump_{T \le \gamma} \phi(\gamma)
	  - \frac{1}{2\pi}\int_T^\infty \phi(t)\log(t/2\pi)\dif t\,.
\end{equation}
Then $E(T) = -\phi(T)Q(T) + E_2(T)$, where
\begin{equation}				\label{eq:ET1}
E_2(T) = -\int_T^\infty\phi'(t)Q(t)\dif t\,,
\end{equation}
and $Q(T) = N(T)-L(T)$ is defined by~\eqref{eq:N}--\eqref{eq:L}.
Also,
\begin{equation}				\label{eq:ET2}
|E_2(T)|
 \le 2(A_0+A_1\log T)\,|\phi'(T)| 
  + (A_1 + A_2)\phi(T)/T.
\end{equation}
\end{theorem}
Here $A_0$ and $A_1$ are constants satisfying
condition~\eqref{eq:S1bound} below, and $A_2$ is a small constant
which, from~Lemma~\ref{lem:Q-S_bound}, we can take as
$A_2 = 1/150$.
We note that $E_2(T)$ is a continuous function of $T$,
as can be seen from~\eqref{eq:ET1},
whereas $E(T)$ has jumps at the ordinates of nontrivial
zeros of $\zeta(s)$.

Disregarding the constant factors,
Theorem~\ref{thm:convergent} shows that
\begin{equation*}
E_2(T) \ll |\phi'(T)|\log T  + \phi(T)/T.
\end{equation*}
For example, if $\phi(t) = t^{-c}$ for some $c > 1$, then
$E(T) \ll T^{-c}\log T$, and
$E_2(T) \ll T^{-(c+1)}\log T$ is smaller by a factor of order~$T$.

As well as convergent sums,
we also consider certain divergent sums.
Theorem~\ref{thm:limit} shows that, if
$\int_{T_0}^\infty t^{-1}\phi(t)\dif t < \infty$,
then there exists
\[
F(T_0) :=
\lim_{T \to \infty} \left(
\sump_{T_0 \le \gamma \le T} \phi(\gamma) - 
 \frac{1}{2\pi}\int_{T_0}^{T}\phi(t)\log(t/2\pi)\dif t \right).
\]
In Theorem~\ref{thm:divergent} we consider approximating $F(T_0)$
by computing a finite sum (over $\gamma \le T$), with error term
$E_2(T)$ the same as in Theorem~\ref{thm:convergent}.

For example, if $\phi(t) = 1/t$ and $T_0 = 2\pi$, we have
$E(T) \ll T^{-1}\log T$ and 
$E_2(T) \ll T^{-2}\log T$.	
The latter bound
allows us to obtain an accurate approximation to the constant
$H = F(2\pi)$ that can equally well be defined,
in analogy to Euler's constant, by
\[
H := \lim_{T\to\infty} \left(
 \sum_{0 < \gamma \le T} \frac{1}{\gamma} -
  \frac{1}{4\pi}\log^2(T/2\pi)\right).
\]
This example is considered in detail in~\cite{BPTHass},
where it is shown that
\[
H = -0.0171594043070981495 + \vartheta(10^{-18}).
\]
The motivation for this paper was an attempt to
generalise the results of~\cite{BPTHass}.

In \S\ref{sec:prelim} we define some notation and mention some relevant
results in the literature.
We also state Lemma~\ref{lem:Q-S_bound}, which sharpens a result
of Trudgian~\cite{Trudgian-2014} and gives an almost best-possible
explicit bound on $Q(t)-S(t)$.
Lemma~\ref{lem:finite} in \S\ref{sec:finite} covers 
finite sums.
In \S\ref{sec:convergent}--\S\ref{sec:divergent}
we deduce Theorems \ref{thm:convergent}--\ref{thm:divergent}
from Lemma~\ref{lem:finite}.
Thus, in a sense, Lemma~\ref{lem:finite} is the key result,
but we have called it a lemma in deference to Lehman's lemma.

\pagebreak[3]

\section{Preliminaries}				\label{sec:prelim}

The \emph{Riemann-Siegel theta function} $\theta(t)$
is defined for real $t$ by
\[
\theta(t) := \arg\Gamma\left(\frac14 + \frac{it}{2}\right)
                - \frac{t}{2}\log\pi,
\]
see for example~\cite[\S6.5]{Edwards}.
The argument is defined so that $\theta(t)$ is continuous on $\R$,
and $\theta(0)=0$. 

Let $\NZ$ denote the set of positive ordinates of zeros of $\zeta(s)$.
Following Titchmarsh~\cite[\S9.2--\S9.3]{Titch}, 
if $0 < T \not\in\NZ$,
then we let $N(T)$ denote the
number of zeros $\beta+i\gamma$ of $\zeta(s)$ with
$0 < \gamma \le T$, and
$S(T)$ denote the value of $\pi^{-1}\arg\zeta(\frac12+iT)$ obtained by
continuous variation along the straight lines joining
$2$, $2+iT$, and $\frac12+iT$, starting with the value~$0$.
If $0 < T \in\NZ$, we take
$S(T) = \lim_{\delta\to 0}[S(T-\delta)+S(T+\delta)]/2$,
and similarly for $N(T)$.
This convention is the reason why we consider sums of the form
$\sumpsmall_{T_1\le\gamma\le T_2}\phi(t)$ instead of
$\sum_{T_1\le\gamma\le T_2}\phi(t)$.

By~\cite[Thm.~9.3]{Titch}, we have
\begin{align}						\label{eq:N}
N(T) &= L(T) + Q(T),\\
							\label{eq:L}
L(T) &= \frac{T}{2\pi}\left(\log\left(\frac{T}{2\pi}\right)-1\right)
        +\frac78\,,\text{ and}\\
				                        \label{eq:S}
S(T) &= Q(T) + O(1/T).
\end{align}
{From}~\cite[Thm.~9.4]{Titch}), $S(T) \ll \log T$.
Thus, from~\eqref{eq:S}, $Q(T) \ll \log T$.

Trudgian~\cite[Cor.~1]{Trudgian-2014} gives the explicit bound
$|Q(T)-S(T)| \le 0.2/T$ for all $T \ge e$. In Lemma~\ref{lem:Q-S_bound}
we obtain a sharper constant, assuming that $T \ge 2\pi$. 
The result of Lemma~\ref{lem:Q-S_bound} is close to optimal, since
the proof shows that the constant $150$ could at best
be replaced by $48\pi \approx 150.8$.

\begin{lemma}					\label{lem:Q-S_bound}
If $Q(t)$ and $S(t)$ are defined as above then,
for all $t \ge 2\pi$, 
\begin{equation*}				\label{eq:Q-S_bound}
|Q(t)-S(t)| \le \frac{1}{150\,t}\,.
\end{equation*}
\end{lemma}

\begin{proof}
We shall assume that $t \not\in\NZ$, since otherwise the result
follows by continuity of $Q(t)-S(t)$.
The Riemann-von Mangoldt formula
states, in its most precise form,
\[
N(t) = {\theta(t)}/{\pi} + 1 + S(t).
\]
{From}~\eqref{eq:N},
this implies that
\[
Q(t)-S(t) = \frac{\theta(t)}{\pi} + 1 - L(t).
\]
Now $\theta(t)$ has a well-known asymptotic
expansion~\cite[Satz 4.2.3(c)]{Gabcke}
\begin{equation}                        \label{eq:theta-asymp}
\theta(t) \sim \frac{t}{2}\left(\log\left(\frac{t}{2\pi}\right)-1\right)
                - \frac{\pi}{8}
                + \sum_{j \ge 1} \frac{(1-2^{1-2j})|B_{2j}|}
                        {4j(2j-1)t^{2j-1}}\,,
\end{equation}
where $B_2 = \frac16, B_4 = -\frac{1}{30},\; \ldots$ are Bernoulli numbers.
Thus, using~\eqref{eq:L},
$Q(t)-S(t)$ has an asymptotic expansion
\begin{equation}                        \label{eq:QS-asymp}
Q(t) - S(t) \sim \frac{1}{\pi}\sum_{j \ge 1} \frac{(1-2^{1-2j})|B_{2j}|}
                        {4j(2j-1)t^{2j-1}}\,.
\end{equation}
In order to give an explicit bound on $Q(t)-S(t)$, we use an explicit
bound on the error incurred by taking the first $k$ terms
$\widetilde{T}_j(t)$, $j = 1,\ldots,k$ in \eqref{eq:theta-asymp}.
{From}~\cite[(47)]{rpb268}, for all $t > 0$, this error is
\begin{equation}                        \label{eq:Rk+1bd}
|\widetilde{R}_{k+1}(t)| < (1-2^{1-2k})^{-1}\,(\pi k)^{1/2}\,
                                \widetilde{T}_k(t) 
                                + {\textstyle\frac12} e^{-\pi t}. 
\end{equation}
Substituting the expression for $\widetilde{T}_k(t)$ into
\eqref{eq:Rk+1bd} gives a bound
\begin{equation*}                       
\frac{|\widetilde{R}_{k+1}(t)|}{\pi}
 < \frac{|B_{2k}|}{4 (\pi k)^{1/2}(2k-1)\,t^{2k-1}}
        + \frac{e^{-\pi t}}{2\pi}
\end{equation*}
for the error incurred by taking the first $k$ terms in
\eqref{eq:QS-asymp}.
Thus, for all $k \ge 1$ and $t > 0$,
\begin{align*}
Q(t) - S(t) = \frac{1}{\pi}\sum_{j=1}^k \frac{(1-2^{1-2j})|B_{2j}|}
                        {4j(2j-1)\,t^{2j-1}}
        + \frac{\vartheta|B_{2k}|}{4(\pi k)^{1/2}(2k-1)\,t^{2k-1}}
        + \frac{\vartheta e^{-\pi t}}{2\pi}\,.
\end{align*}
Taking $k=3$ and using the assumption $t \ge 2\pi$, we obtain the result.
\end{proof}

Define $S_1(T) := \int_0^TS(t)\dif t$.
We know 
that $S_1(T) \ll \log T$,
and that $S_1(T) = o(\log T)$ if and only if the
Lindel\"of Hypothesis is true~---
see Titchmarsh~\cite[Thm.~9.9(A), Thm.~13.6(B), and Note 13.8]{Titch}.

Explicit bounds on $S_1(T)$ are known
\cite{Edwards,Trudgian-2011,Trudgian-2016,Turing}.
{From}~\cite[Thm~2.2]{Trudgian-2011},
\begin{equation}                        \label{eq:S1bound}
|S_1(T)-c_0| \le A_0 + A_1\log T \text{ for all } T \ge 168\pi,
\end{equation}
where $c_0 = S_1(168\pi)$,
$A_0=2.067$, and $A_1=0.059$.
However, a small computation shows
that~\eqref{eq:S1bound} also holds for $T \in [2\pi,168\pi]$.
Hence, from now on we assume that $T_0 \ge 2\pi$ and that
\eqref{eq:S1bound} holds for $T \ge T_0$.

\section{Finite sums}				\label{sec:finite}

In this section we prove Lemma~\ref{lem:finite}, which may be seen
as a refinement of Lemma~\ref{lem:Lehman} if the conditions
$\phi'(t) \le 0$, $\phi''(t) \ge 0$ are satisfied.
The proof of Lemma~\ref{lem:finite} is essentially the same as
the proof of Lehman's lemma up to equation~\eqref{eq:alpha}, but then
differs in the way that $\int_{T_1}^{T_2}\phi'(t)Q(t)\dif t$ is bounded.

{From} the discussion in \S\ref{sec:prelim}, we may assume that the
constants $A_0, A_1, A_2$ occurring in Lemma~\ref{lem:finite} are
$A_0 = 2.067$, $A_1 = 0.059$, and $A_2 = 1/150 < 0.007$.
The first two values could probably be improved significantly.


\begin{lemma}					\label{lem:finite}
If $2\pi \le T_0 \le T_1 \le T_2$ and
\begin{equation*}				\label{eq:ET1T2}
E(T_1,T_2) := \sump_{T_1 \le \gamma \le T_2} \phi(\gamma)
	  - \frac{1}{2\pi}\int_{T_1}^{T_2} \phi(t)\log(t/2\pi)\dif t\,,
\end{equation*}
then $E(T_1,T_2) = \phi(T_2)Q(T_2)-\phi(T_1)Q(T_1) + E_2(T_1,T_2)$,
where
\begin{equation}
E_2(T_1,T_2) = -\int_{T_1}^{T_2}\phi'(t)Q(t)\dif t\,,	\label{eq:T1T2}
\end{equation}
and
\begin{align}
|E_2(T_1,T_2)|
 &\le 2(A_0+A_1\log T_1)\,|\phi'(T_1)| 		\label{eq:E2T1T2}
 + (A_1 + A_2)\phi(T_1)/T_1.
\end{align}
\end{lemma}
\begin{proof}
Assume initially that $T_1\not\in\NZ$, $T_2\not\in\NZ$.
Using Stieltjes integrals,
we see that
\begin{align*}
\sump_{T_1 \le \gamma \le T_2}\phi(\gamma) 
 &= \int_{T_1}^{T_2} \phi(t)\dif N(t)
  = \int_{T_1}^{T_2} \phi(t)\dif L(t) 
  + \int_{T_1}^{T_2} \phi(t)\dif Q(t)\\
 &= \frac{1}{2\pi}\int_{T_1}^{T_2} \phi(t)\log(t/2\pi)\dif t 
    + \int_{T_1}^{T_2} \phi(t)\dif Q(t)\,,
\end{align*}
so
\begin{align}
E(T_1,T_2) &= \int_{T_1}^{T_2} \phi(t)\dif Q(t)
 = \left[\phi(t)Q(t) - \int\phi'(t)Q(t)\dif t\right]_{T_1}^{T_2}
							\nonumber \\
 &= \phi(T_2)Q(T_2) - \phi(T_1)Q(T_1) 
  - \int_{T_1}^{T_2} \phi'(t)Q(t)\dif t\,.		\label{eq:alpha}
\end{align}
This proves~\eqref{eq:T1T2}.
To prove~\eqref{eq:E2T1T2}, note that,
from~\eqref{eq:S} and Lemma~\ref{lem:Q-S_bound},
\begin{equation}					\label{eq:beta}
\int_{T_1}^{T_2}\phi'(t)Q(t)\dif t
  = \int_{T_1}^{T_2}\phi'(t)S(t)\dif t
    + \vartheta A_2\int_{T_1}^{T_2}\frac{\phi'(t)}{t}\dif t,
\end{equation}
and the last integral can be bounded using
\begin{equation}					\label{eq:gamma}
\left|\int_{T_1}^{T_2}\frac{\phi'(t)}{t}\dif t\right|
 \le \frac{1}{T_1}\int_{T_1}^{T_2} |\phi'(t)|\dif t
 = \frac{\phi(T_1)-\phi(T_2)}{T_1}
 \le \frac{\phi(T_1)}{T_1}\,. 
\end{equation}\
Also,
\begin{align*}
&\int_{T_1}^{T_2}\phi'(t)S(t)\dif t
 = \left[\phi'(t)(S_1(t)-c_0) 
   - \int\phi''(t)(S_1(t)-c_0)\dif t\right]_{T_1}^{T_2}\\
 =&\;\phi'(T_2)(S_1(T_2)-c_0)
   -\phi'(T_1)(S_1(T_1)-c_0)
    -\int_{T_1}^{T_2}\phi''(t)(S_1(t)-c_0)\dif t.
\end{align*}
Now, using $\phi'(t) \le 0$ and
$|S_1(t)-c_0| \le A_0 + A_1\log t$, 
we have
\[
|\phi'(t)(S_1(t)-c_0)| \le -(A_0+A_1\log t)\phi'(t)
\]
for $t = T_1, T_2$.
Thus
\begin{align}
&\left|\int_{T_1}^{T_2}\phi'(t)S(t)\dif t\right|	\nonumber\\
&\;\;\le -\sum_{j=1}^2 (A_0+A_1\log T_j)\phi'(T_j)
 + \left|\int_{T_1}^{T_2}\phi''(t)(S_1(t)-c_0)\dif t\right|.
							\label{eq:sigmaj}
\end{align}
Also, using 
$\phi''(t) \ge 0$,
we have
\begin{align}
&\left|\int_{T_1}^{T_2}\phi''(t)(S_1(t)-c_0)\dif t\right|
 \le A_0\int_{T_1}^{T_2}\phi''(t)\dif t + A_1\int_{T_1}^{T_2}\phi''(t)\log t\dif t
							\nonumber \\
&= A_0(\phi'(T_2)-\phi'(T_1))
 + A_1\left[\phi'(t)\log t - \int\frac{\phi'(t)}{t}\dif t\right]_{T_1}^{T_2}
							\nonumber \\
&= (A_0+A_1\log T_2)\phi'(T_2)
   -(A_0+A_1\log T_1)\phi'(T_1)
   - A_1\int_{T_1}^{T_2}\frac{\phi'(t)}{t}\dif t. 
							\label{eq:eta}
\end{align}
Inserting \eqref{eq:eta} in \eqref{eq:sigmaj} and simplifying, terms
involving $T_2$ cancel, giving
\begin{equation}					\label{eq:phipS}
\left|\int_{T_1}^{T_2}\phi'(t)S(t)\dif t\right| \le
 -2(A_0+A_1\log T_1)\phi'(T_1) - A_1\int_{T_1}^{T_2}\frac{\phi'(t)}{t}\dif t.
\end{equation}
Combining \eqref{eq:T1T2} with
\eqref{eq:beta}, \eqref{eq:gamma}, and \eqref{eq:phipS},
gives~\eqref{eq:E2T1T2}.
Finally, we note that~\eqref{eq:T1T2}--\eqref{eq:E2T1T2} hold even if
$T_1\in\NZ$ and/or $T_2\in\NZ$, because of the way that we defined
$N(T)$ (and hence $Q(T)=N(T)-L(T)$) for $T\in\NZ$.
\end{proof}

\begin{remark}					\label{rem:compare-Lehman}
With the assumptions and notation of Lemma~\ref{lem:finite},
Lemma~\ref{lem:Lehman} gives the bound
\begin{equation}
|E(T_1,T_2)| \le 
 A \left(2\phi(T_1)\log T_1 
 + \int_{T_1}^{T_2}\frac{\phi(t)}{t}\dif t\right)\,.	\label{eq:E-bound}
\end{equation}
Our bound~\eqref{eq:E2T1T2} on $E_2(T_1,T_2)$ is often smaller
than the bound~\eqref{eq:E-bound} on $E(T_1,T_2)$.
We can take advantage of this if the terms
$\phi(T_j)Q(T_j)$ ($j=1,2$) are known.
Examples are given in \S\S\ref{sec:convergent}--\ref{sec:divergent}.
\end{remark}

\section{Convergent sums}				\label{sec:convergent}

In this section we assume that
$\sum_{T \leq \gamma}\phi(\gamma) < \infty,$
or equivalently
(given our conditions on $\phi$), that
$\int_{T}^\infty \phi(t)\log(t/2\pi)\dif t < \infty.$
We first state an easy lemma, and then prove Theorem~\ref{thm:convergent}.

\begin{lemma}						\label{lem:A}
Suppose that $2\pi \le T_0 \le T$ and
$\int_T^\infty\phi(t)\log(t/2\pi)\dif t < \infty$.
Then 
\begin{align}
\phi(t)\log t  = o(1) \text{ as } t \to &\;\infty,	\label{eq:A1}\\
\phi'(t)\log t = o(1) \text{ as } t \to &\;\infty,
	\text{ and } 					\label{eq:A2}\\
\textstyle \int_T^\infty |\phi'(t)|\log t \dif t < &\;\infty. 
							\label{eq:A3}
\end{align}
\end{lemma}

\begin{proof}
For $u \ge T$, 
\[
\int_u^{u+1}\phi(t)\log(t/2\pi)\dif t \ge \phi(u+1)\log(u/2\pi).
\]
Thus $\phi(u+1)\log(u/2\pi) = o(1)$ as $u \to \infty$,
and $\phi(t)\log((t-1)/2\pi) = o(1)$.
Since $\log((t-1)/2\pi)\sim \log t$, \eqref{eq:A1} follows.

For \eqref{eq:A2}, we have
\begin{equation}			\label{eq:phi_phip}
\phi(u) \ge \phi(u) - \phi(u+1) = \int_u^{u+1}|\phi'(t)|\dif t
 \ge |\phi'(u+1)|,
\end{equation}
so \eqref{eq:A1} implies that 
$\phi'(u+1)\log u = o(1)$. Taking $t = u+1$, we have
$\phi'(t)\log(t-1) = o(1)$.
Since $\log(t-1)\sim \log t$, \eqref{eq:A2} follows.

Finally, from~\eqref{eq:phi_phip}, $|\phi'(t)| \le \phi(t-1)$
for $t \ge T+1$, so
\[
\int_{T+1}^\infty|\phi'(t)|\log t \dif t
\le \int_{T+1}^\infty \phi(t-1)\log t \dif t
\ll \int_T^\infty \phi(t)\log(t/2\pi) \dif t < \infty.
\]
and~\eqref{eq:A3} follows.
\end{proof}

\begin{proof}[Proof of Theorem~$\ref{thm:convergent}$]
We have $\phi(t)\log t = o(1)$ by 
Lemma~\ref{lem:A} and
convergence of the integral in~\eqref{eq:ET0}.
Also, from Lemma~\ref{lem:A}
we have
\hbox{$\int_{T}^\infty |\phi'(t)|\log t\dif t < \infty,$}
but $Q(t) \ll \log t$, so 
$
\int_T^\infty \phi'(t)Q(t)\dif t
$
converges absolutely.
Now, Lemma~\ref{lem:finite} gives
\begin{align}
&\sump_{T\le\gamma\le T_2} \phi(\gamma)
 - \frac{1}{2\pi}\int_{T}^{T_2}\phi(t)\log(t/2\pi)\dif t \nonumber\\
&\;\;\;\; = \phi(T_2)Q(T_2)-\phi(T)Q(T)
   		- \int_{T}^{T_2}\phi'(t)Q(t)\dif t. \label{eq:finite}
\end{align}
If we let $T_2 \to \infty$ in \eqref{eq:finite}, 
$\phi(T_2)Q(T_2) \to 0$ and
$\int_{T}^{T_2}\phi'(t)Q(t)\dif t$ tends to a finite limit. 
Thus, the right side of~\eqref{eq:finite} tends to a finite limit,
and the left side must tend to the same limit. 
This gives
\[
\sump_{T\le\gamma} \phi(\gamma)
 - \frac{1}{2\pi}\int_{T}^{\infty}\phi(t)\log(t/2\pi)\dif t
 = -\phi(T)Q(T) - \int_{T}^\infty \phi'(t)Q(t)\dif t.
\]
We have proved \eqref{eq:ET0}--\eqref{eq:ET1}
of Theorem~\ref{thm:convergent}.
The bound \eqref{eq:ET2} follows by observing that the bound 
\eqref{eq:E2T1T2} of Lemma~\ref{lem:finite} is independent of $T_2$,
so 
\[
\left|\int_T^\infty \phi'(t)Q(t)\dif t\right|
 \le 2(A_0+A_1\log T)\,|\phi'(T)| + (A_1 + A_2)\phi(T)/T.
\]
This completes the proof of Theorem~\ref{thm:convergent}.
\end{proof}

\begin{example}						\label{ex:c1}
We consider computation of the constant
\begin{equation*}                        		\label{eq:c1}
c_1 := \sum_{\gamma > 0}\frac{1}{\gamma^2} 
  = 0.02310499\ldots.
\end{equation*}
The approximation $0.023105$ was given
in~\cite[Lemma 2.9]{Saouter}, where it was computed
using a finite sum with (essentially) Lemma~\ref{lem:Lehman} 
to bound the tail.

Taking $\phi(t) = 1/t^2$ in Lemma~\ref{lem:Lehman} gives an error term
\[
|E(T)| \le A\left(\frac{\frac12 + 2\log T}{T^2}\right)
= \frac{0.14 + 0.56\log T}{T^2}\,,
\]
using the value $A=0.28$ mentioned above.
The corresponding error term given by Theorem~\ref{thm:convergent} is
\[
|E_2(T)| \le \frac{(4A_0+A_1+A_2) + 4A_1\log T}{T^3}
	\le \frac{8.334 + 0.236\log T}{T^3}\,,
\]
using the values of $A_0, A_1, A_2$ above.
For example, taking $T = 1000$ (corresponding to the first $649$ nontrivial
zeros), we get $|E(T)| \le 4.009\times 10^{-6}$
and $|E_2(T)| \le 9.965\times 10^{-9}$,
an improvement by a factor of $400$.
If we use $10^{10}$ zeros, as in Corollary~\ref{cor:huck},
the improvement is by a factor of $3\times 10^9$.
\end{example}
\begin{corollary}\label{cor:huck}
We have
\[
c_1 = \sum_{\gamma > 0}\frac{1}{\gamma^2} =
0.0231049931154189707889338104 + \vartheta(5\times 10^{-28}).
\]
\end{corollary}
\begin{proof}
This follows from Theorem~\ref{thm:convergent} by an interval-arithmetic
computation using the first $n = 10^{10}$ zeros, 
with $T = 3293531632.542\cdots \in (\gamma_n,\gamma_{n+1})$.
\end{proof}

\begin{remark}				\label{rem:Catalan}
Assuming the Riemann Hypothesis (RH),
there is an equivalent expression:\footnote{The
formula \eqref{eq:Catalan-c3} is stated 
in~\cite[(21)]{Guillera} and
is proved in \cite[p.~13]{Arias130802}.
An almost indecipherable sketch of this result may be found in 
Riemann's Nachlass.}
\begin{equation}                        \label{eq:Catalan-c3}
c_1 =   \frac{d^2\log\zeta(s)/ds^2|_{s=1/2}}{2}
        + \frac{\pi^2}{8} + G - 4,
\end{equation}
where $G=\beta(2)$ 
is Catalan's constant $0.915965\cdots$.
%
This enables us to confirm  Corollary~\ref{cor:huck}
without summing over any zeros of $\zeta(s)$, but assuming RH.
It is only rarely that such a closed form is known. 
One other example is the following~---
see, e.g., \cite[Ch.~12]{Davenport}. 
Assuming RH, we have 
\begin{equation*}\label{porto}
\sum_{\gamma>0} \frac{1}{\gamma^{2} + \frac{1}{4}} = 
\sum_{\rho} \Re \left(\frac{1}{\rho}\right) = 
 1 + \frac{C}{2} - \frac{\log 4\pi}{2} = 0.0230957\ldots,
\end{equation*}
where $C=0.5772\ldots$ is Euler's constant.
\end{remark}

\section{Divergent sums}			\label{sec:divergent}

In this section we give two theorems that apply, subject to a mild
condition~\eqref{eq:mildphi} on $\phi(t)$, 
even if $\sum_{T \leq \gamma}\phi(\gamma)$
diverges. Theorem~\ref{thm:limit} shows the existence of a limit for the
difference between a sum and the corresponding integral.
Theorem~\ref{thm:divergent} shows how we can accurately approximate
the limit. 

First we prove two lemmas that strengthen the first and third parts
of Lemma~\ref{lem:A}.
In Lemma~\ref{lem:C}, $f$ is non-increasing but need not be differentiable.

\begin{lemma}                                   \label{lem:C}
Suppose that, for some $T \ge 1$, $f:[T,\infty]\mapsto[0,\infty)$
is non-negative and non-increasing on $[T,\infty)$.
If
\begin{equation}				\label{eq:fovert}
\int_{T}^\infty \frac{f(t)}{t}\dif t < \infty,
\end{equation}
then $f(t)\log t = o(1)$.
\end{lemma}
\begin{proof}
Assume, by way of contradiction, that $f(t) \log t \ne o(1)$. Thus, there exists a constant $c > 0$
and an unbounded increasing sequence $(t_n)_{n\ge 1}$ such that
$t_1 > T$ and
\begin{equation}                                \label{eq:fndef}
f_n := f(t_n) \ge \frac{c}{\log t_n}\,.
\end{equation}
Moreover,
by taking a subsequence of $(t_n)_{n\ge 1}$ if necessary, we can assume that
$t_{n+1} \ge t_n^2$ for all $n \ge 1$.  
Thus
\begin{equation}                                \label{eq:logineq}
\log\left(\frac{t_{n+1}}{t_n}\right) \ge \frac{\log t_{n+1}}{2}\,.
\end{equation}
Since $f(t)$ is non-increasing, we have
$f(t) \ge f_{n+1}$ on $[t_n,t_{n+1}]$, and
\[
\int_{t_n}^{t_{n+1}}\frac{f(t)}{t}\dif t
 \ge \int_{t_n}^{t_{n+1}}\frac{f_{n+1}}{t}\dif t
 = f_{n+1}\log\left(\frac{t_{n+1}}{t_n}\right).
\]
Using \eqref{eq:fndef}--\eqref{eq:logineq}, this gives
\[
\int_{t_1}^{t_{n+1}} \frac{f(t)}{t}\dif t
 \ge \frac{1}{2} \sum_{k=1}^n f_{k+1}\log t_{k+1}
 \ge \frac{c}{2} \sum_{k=1}^n 1 = \frac{cn}{2} \to \infty.
\]
This contradicts the condition~\eqref{eq:fovert}.
Thus, our assumption 
is false, and we must have $f(t)\log t = o(1)$.
\end{proof}

\pagebreak[3]
\begin{lemma}                                   	\label{lem:D}
If $\int_{T_0}^\infty \frac{\phi(t)}{t}\dif t < \infty$, then
\hbox{$\int_{T_0}^\infty\phi'(t)\log t\dif t$} is absolutely convergent.
\end{lemma}
\begin{proof}
For $T \ge T_0$ we have
\begin{equation}					\label{eq:phip_phi}
\int_{T_0}^T\phi'(t)\log t\dif t = \phi(T)\log T - \phi(T_0)\log T_0
		- \int_{T_0}^T \frac{\phi(t)}{t}\dif t.
\end{equation}
As $T \to \infty$ in \eqref{eq:phip_phi}, the term $\phi(T)\log T \to 0$
by Lemma~\ref{lem:C}, and the integral on the right-hand side tends to a
finite limit.  Thus, the integral on the left-hand side tends to a finite
limit. Since $\phi'(t)\log t \le 0$ has constant sign on $[T_0,\infty)$,
the integral is absolutely convergent.
\end{proof}

\begin{theorem}					\label{thm:limit}
Suppose that $T_0 \ge 2\pi$, and
\begin{equation}				\label{eq:mildphi}
\int_{T_0}^\infty \frac{\phi(t)}{t}\dif t < \infty.
\end{equation}
Then there exists
\[
F(T_0) := 
\lim_{T\to\infty}\left(\sump_{T_0 \le \gamma \le T}\phi(\gamma)
  - \frac{1}{2\pi}\int_{T_0}^{T}\phi(t)\log(t/2\pi)\dif t\right),
\] 
and
\begin{equation}				\label{eq:limit_integral}
F(T_0) = -\phi(T_0)Q(T_0) - \int_{T_0}^\infty\phi'(t)Q(t)\dif t.
\end{equation}
\end{theorem}

\begin{proof}
Suppose that $T \ge T_0$.
Applying Lemma~\ref{lem:finite}, we have
\begin{align}
\sump_{T_0\le\gamma\le T}\phi(\gamma)
 &- \frac{1}{2\pi}\int_{T_0}^T \phi(t)\log(t/2\pi)\dif t \nonumber\\
 &= \phi(T)Q(T)-\phi(T_0)Q(T_0) - \int_{T_0}^T\phi'(t)Q(t)\dif t.
							\label{eq:dif2}
\end{align}
Let $T\to\infty$ in \eqref{eq:dif2}. 
On the right-hand side,
$\phi(T)Q(T) \to 0$ by Lemma~\ref{lem:C}, and the integral tends to a
finite limit by Lemma~\ref{lem:D}, using $Q(t) \ll \log t$.
Thus the left-hand side tends to a finite
limit $F(T_0)$.
This gives \eqref{eq:limit_integral}.
\end{proof}

The identity~\eqref{eq:limit_integral} is not convenient for accurately
approximating $F(T_0)$ when $T_0$ is small, because
$\int_{T_0}^\infty\phi'(t)Q(t)\dif t$ is not necessarily small.
In Theorem~\ref{thm:divergent} we use a finite sum
(over $\gamma \le T$)
and integral to approximate $F(T_0)$.
Theorem~\ref{thm:divergent}
has the same expression for the error
term $E_2$ as Theorem~\ref{thm:convergent},
essentially because the bounds in both theorems 
are proved using Lemma~\ref{lem:finite}.

\begin{theorem}					\label{thm:divergent}
Suppose that $2\pi\le T_{0} \le T_{1}$ and
$\phi(t)$ satisfies~\eqref{eq:mildphi}.
Let 
\[
F(T_0) := 
\lim_{T\to\infty}\Bigg(\sump_{T_0 \le \gamma \leq T}\phi(\gamma)
  - \frac{1}{2\pi}\int_{T_0}^{T}\phi(t)\log(t/2\pi)\dif t\Bigg).
\] 
Then
\[
F(T_0) = \sump_{T_0 \le \gamma \le T_1}\phi(\gamma)
  - \frac{1}{2\pi}\int_{T_0}^{T_1}\phi(t)\log(t/2\pi)\dif t
  - \phi(T_1)Q(T_1) + E_2(T_1),
\] 
where 
$E_2(T_1) = -\int_{T_1}^\infty\phi'(t)Q(t)\dif t$,
and 
\[
|E_2(T_1)|
 \le 2(A_0+A_1\log T_1)\,|\phi'(T_1)| 
  + (A_1 + A_2)\phi(T_1)/T_1.
\]
\end{theorem}
\begin{proof}
We note that, from
Theorem~\ref{thm:limit}, the limit defining $F(T_0)$ exists.
Also, from Lemmas~\ref{lem:C}--\ref{lem:D}, \hbox{$\phi(T)Q(T) = o(1)$} and
\hbox{$\int_{T_0}^\infty\phi'(t)Q(t)\dif t < \infty$}.
Thus, using Lemma \ref{lem:finite} as in the proof of
Theorem~\ref{thm:convergent}, we see that
\begin{align*}
\lim_{T_2\to\infty}
\Bigg(
\sump_{T_1\le\gamma\le T_2}\phi(\gamma)
  - \frac{1}{2\pi}\int_{T_1}^{T_2}&\phi(t)\log(t/2\pi)\dif t\Bigg)
							\nonumber\\
 &= -\phi(T_1)Q(T_1) - \int_{T_1}^\infty\phi'(t)Q(t)\dif t
							\label{eq:limT2}
\end{align*}
and
$\displaystyle\;\;
\left|\int_{T_1}^\infty\phi'(t)Q(t)\dif t\right|
\le 2(A_0+A_1\log T_1)\,|\phi'(T_1)| + (A_1 + A_2)\phi(T_1)/T_1.
$
Since
\begin{align*}
F(T_0) = 
\lim_{T_2\to\infty}\Bigg(\sump_{T_1\le\gamma\le T_2}\phi(\gamma)
 &-\; \frac{1}{2\pi}\int_{T_1}^{T_2}\phi(t)\log(t/2\pi)\dif t\Bigg)\\
 &+ \sump_{T_0\le\gamma\le T_1}\phi(\gamma)
  - \frac{1}{2\pi}\int_{T_0}^{T_1}\phi(t)\log(t/2\pi)\dif t,
\end{align*}
the result follows.
\end{proof}

\begin{example}				\label{ex:rapid-divergence}
To illustrate the divergent case, we consider the example 
$\phi(t) = 1/(\log(t/2\pi))^2$.
The constant $2\pi$ here is unimportant,
but this choice simplifies some of the expressions below.

{From} Lemma~\ref{lem:Lehman}, the asymptotic behaviour of
$\sum_{0<\gamma\le T}\phi(\gamma)$ is given by
\[
\frac{1}{2\pi}\int_c^T\phi(t)\log(t/2\pi)\,\dif t
= \li(T/2\pi) - \li(c/2\pi)
\sim \frac{T}{2\pi\log T}\,,
\]
where $c \ge 2\pi e$ is an arbitrary constant,
and $\li(x)$ is the logarithmic integral, defined 
in the usual way by a principal value integral.
This motivates the definition of a constant $c_2$ by
\begin{equation}				\label{eq:c2}
c_2 := \lim_{T\to\infty}\Bigg(\,
 \sump_{0 < \gamma \le T}\phi(\gamma) - \li(T/2\pi)\Bigg),
\end{equation}
where the limit exists by Theorem~\ref{thm:limit}.

If we use~\eqref{eq:c2} to estimate $c_2$ then, by 
Theorem~\ref{thm:divergent}, the error is
\[
E(T) = -\phi(T)Q(T) + O(|\phi'(T)|\log T) + O(\phi(T)/T)
	\ll \frac{1}{\log T}\,.
\]
Convergence is so slow that it is difficult to obtain more than
two correct decimal digits.
On the other hand, if we estimate $c_2$ using the approximation
\begin{equation}				\label{eq:c2fast}
\sump_{0 < \gamma \le T}\phi(\gamma) - \li(T/2\pi) - \phi(T)Q(T)
\end{equation}
suggested by Theorem~\ref{thm:divergent},
then the error is $E_{2}(T) \ll ( T \log^{2} T)^{-1}$,
smaller by a factor of order $T\log T$.
More precisely, from Theorem~\ref{thm:divergent} we have
\begin{align}				\nonumber
|E_2(T)| &\le \frac{4(A_0+A_1\log T)}{T\log^3(T/2\pi)}
	    + \frac{A_1+A_2}{T\log^2(T/2\pi)}\\
					\label{eq:E2ex2}
&\le \frac{0.302\log(T/2\pi)+8.702}{T\log^3(T/2\pi)}\,.
\end{align}

\begin{corollary}
If $c_2$ is defined by \eqref{eq:c2}, then
\begin{equation*}				\label{eq:c2est1e-7}
c_2 = -0.5276697875 + \vartheta(10^{-10}).
\end{equation*}
\end{corollary}
\begin{proof}
Using the first $n=10^9$ nontrivial zeros with
$T \approx (\gamma_n+\gamma_{n+1})/2$
in~\eqref{eq:c2fast}, and the error bound~\eqref{eq:E2ex2}, 
an interval-arithmetic computation gives the result.
\end{proof}

To illustrate the speed of convergence,
in Table~\ref{tab:c2est_midpt} we give the estimates of $c_2$ obtained from
\eqref{eq:c2} and \eqref{eq:c2fast} by summing over the first $n$ nontrivial
zeros, and the error bound \eqref{eq:E2ex2},
with $T = (\gamma_n + \gamma_{n+1})/2$.
The first incorrect digit in each entry is underlined.

\begin{table}[h]
\begin{center}
\begin{tabular}{ | c | c | c| c | } \hline
  $n$ & estimate via \eqref{eq:c2} &
	estimate via \eqref{eq:c2fast} &
	$|E_2|$ bound \eqref{eq:E2ex2} \\ \hline
$10$ & -0.\ub{4}9986259 & -0.527\ub{3}3908 & $1.96\times 10^{-2}$\\
$10^2$ & -0.5\ub{4}054724 & -0.5276\ub{7}238 & $8.64\times 10^{-4}$\\
$10^3$ & -0.52\ub{2}44974 & -0.5276\ub{7}173 & $4.58\times 10^{-5}$\\
$10^4$ & -0.5\ub{3}117846 & -0.527669\ub{8}0 & $2.78\times 10^{-6}$\\
$10^5$ & -0.5\ub{3}026260 & -0.5276697\ub{7} & $1.87\times 10^{-7}$\\
\hline
\end{tabular}
\end{center}
\vspace*{-10pt}
\caption{Numerical estimation of $c_2$.}
\vspace*{0pt}
\label{tab:c2est_midpt}
\end{table}
\end{example}

\subsection*{Acknowledgements}
We are indebted to Juan Arias de Reyna for 
information on the identity~\eqref{eq:Catalan-c3}, 
and for his translation of the relevant page from
Riemann's Nachlass.
DJP is supported by ARC Grant DP160100932 and EPSRC Grant EP/K034383/1;
TST is supported by ARC Grants DP160100932 and FT160100094.

\pagebreak[3]

\end{document}